\documentclass[11pt]{amsart}

\usepackage[T1]{fontenc}
\usepackage[utf8]{inputenc}
\usepackage{amsmath,amssymb,amsthm}
\usepackage{geometry}
\theoremstyle{plain}
\newtheorem{theorem}{Theorem}[section]
\newtheorem{proposition}[theorem]{Proposition}
\newtheorem{lemma}[theorem]{Lemma}
\newtheorem{corollary}[theorem]{Corollary}
\theoremstyle{definition}
\newtheorem{definition}[theorem]{Definition}
\newtheorem{example}[theorem]{Example}
\theoremstyle{remark}
\newtheorem{remark}[theorem]{Remark}

\newcommand{\Der}{\operatorname{Der}}
\newcommand{\Rad}{\operatorname{Rad}}
\newcommand{\Ann}{\operatorname{Ann}}
\newcommand{\Aext}[1]{\textstyle\bigwedge_{A}^{#1}}
\newcommand{\ad}{\operatorname{ad}}

\begin{document}

\title[Anchor Radicals and Metric Rigidity in $n$-Lie--Rinehart Algebras]
{Anchor Radicals and Metric Rigidity\\ in $n$-Lie--Rinehart Algebras}

\author{Irmely Gladesh Mabanza Nsiloulou}
\address{Department of Mathematics, Faculty of Science and Technology,
Marien Ngouabi University, Brazzaville, Congo}
\email{irmelygladesh@gmail.com}

\author{Basile Guy Richard Bossoto}
\address{Department of Mathematics, Faculty of Science and Technology,
Marien Ngouabi University, Brazzaville, Congo}
\email{basile.bossoto@umng.cg}


\begin{abstract}
For an $n$-Lie--Rinehart algebra, $n>2$, the anchor is defined on
$\Aext{n-1}L$, so its kernel does not determine a distinguished submodule
of $L$. We introduce the anchor radical $\Rad(\rho)\subset L$ and prove
that it is an adjoint-stable ideal on which the bracket is $A$-multilinear.
We then derive an explicit obstruction to the $A$-linearity of metric
compatibility in fundamental directions. This obstruction vanishes on
$\Rad(\rho)$, whereas a strongly nondegenerate compatible metric on all of
$L$ forces the anchor to vanish. These results lead to a natural notion of
orthogonal structure on the anchor radical. For anchors induced by Nambu
tensors, the radical is identified with the cotangent annihilator of the
tensor.
\end{abstract}

\maketitle

\section{Introduction}

Lie--Rinehart algebras provide an algebraic setting in which a Lie bracket
on a module is coupled to derivations of a commutative algebra. Their
$n$-ary analogues combine this principle with Filippov's $n$-Lie algebras,
replacing the binary bracket by an alternating $n$-bracket satisfying the
fundamental identity. In this setting, the natural infinitesimal operators
are indexed not by single elements but by fundamental objects. This shift
is already visible for an $n$-Lie algebra $L$: an element
$X=x_1\wedge\cdots\wedge x_{n-1}$ determines an adjoint operator $\ad_X$,
and the fundamental identity governs the commutator of these operators.
The associated Lie-algebraic structure of fundamental objects was studied
in our earlier work \cite{BOO}. It remains a useful organizing principle
when a commutative base algebra and an anchor are introduced.

For an $n$-Lie--Rinehart algebra $(L,A,[\,\cdot\,,\ldots,\cdot\,],\rho)$,
the anchor has the form
\[
\rho:\Aext{n-1}L \longrightarrow \Der_K(A).
\]
This feature creates a structural difference from ordinary Lie--Rinehart
theory. When $n=2$, the kernel of the anchor is a submodule of $L$ and can
therefore be studied directly as an intrinsic part of the underlying
module. For $n>2$, however, $\ker\rho$ lies in $\Aext{n-1}L$. It records
which fundamental objects act trivially on the base algebra, but it does
not single out those elements of $L$ that are invisible to the anchor in
every possible fundamental direction. In particular, $\ker\rho$ is not
itself a natural domain for a bilinear form on $L$.

This distinction becomes important when one asks for metric or orthogonal
structures. For an ordinary metric $n$-Lie algebra, invariance of a
symmetric bilinear form is expressed through the adjoint action. In the
Rinehart setting the corresponding identity contains the anchor derivative
of the coefficient-valued metric. Consequently, metric compatibility is
intertwined with the $A$-module structure, and the dependence on each
factor of a fundamental object must be controlled. A direct extension of
the usual metric condition to all of $L$ is therefore substantially more
restrictive than in the anchor-free case.

The aim of this paper is to isolate the intrinsic part of $L$ on which this
obstruction disappears. We define the anchor radical by
\[
\Rad(\rho)=\{u\in L: \rho(u\wedge x_2\wedge\cdots\wedge x_{n-1})=0
\text{ for all } x_2,\ldots,x_{n-1}\in L\}.
\]
Thus $\Rad(\rho)$ is an elementwise counterpart of the kernel of the
anchor: instead of asking that one fundamental object belong to
$\ker\rho$, one asks that every fundamental object containing $u$ act
trivially on $A$. This definition is also naturally related to the
fundamental-object viewpoint developed for $n$-Lie algebras in \cite{BOO},
but here the relevant representation is the anchor representation rather
than only the adjoint one.

The first main result shows that the anchor radical is not merely an
$A$-submodule. Using the representation identity of an $n$-Lie--Rinehart
algebra alone, we prove that it is stable under the adjoint action and
hence is an ideal (Proposition~\ref{prop:adjoint-stable}). Moreover, the
Leibniz correction terms disappear on the radical, so the restricted
bracket is $A$-multilinear and $\Rad(\rho)$ becomes an $n$-Lie algebra over
$A$ (Remark~\ref{rem:A-multilinear}). This gives a canonical algebraic
locus inside $L$ selected solely by the anchor.

The second issue is metric compatibility. For a symmetric $A$-bilinear
form on $L$, we compute explicitly the change of the metric-compatibility
operator when one factor of a fundamental object is multiplied by an
element of $A$. The resulting formula is governed by anchor terms
involving the metric variables. These terms vanish when the metric
variables lie in $\Rad(\rho)$, which provides an intrinsic reason for
placing an orthogonal structure on the radical rather than imposing it
globally. Conversely, under finite projectivity and strong
nondegeneracy, compatibility on all of $L$ forces $\rho=0$
(Theorem~\ref{thm:rigidity}). This rigidity phenomenon is the main
obstruction result of the paper.

Recent developments on representations, extensions and crossed modules of
$n$-Lie--Rinehart algebras \cite{BCEM}, as well as on morphisms,
comorphisms and Nambu--Poisson applications \cite{BCZ}, provide a broader
framework in which these questions arise. Our purpose here is different
and complementary: we focus on the degeneracy carried by the anchor itself
and on its consequences for invariant bilinear forms. In the geometric
case of an anchor induced by an $n$-vector $P$, we finally show that the
anchor radical coincides with the cotangent annihilator $\Ann(P)$. Thus
the algebraic construction has a direct geometric interpretation for
Nambu-type anchors.

The paper is organized as follows. Section~\ref{sec:prelim} recalls only
the $n$-Lie and $n$-Lie--Rinehart identities needed later.
Section~\ref{sec:radical} develops the anchor radical and its structural
properties. Section~\ref{sec:rigidity} treats orthogonal structures, the
metric obstruction and the rigidity theorem. Section~\ref{sec:nambu} gives
the Nambu--Poisson interpretation.

\section{$n$-Lie--Rinehart algebras and fundamental objects}
\label{sec:prelim}

Throughout the paper, $K$ denotes a field of characteristic zero and $A$ a
commutative associative $K$-algebra.

\begin{definition}
An \emph{$n$-Lie algebra} is a $K$-vector space $L$ endowed with an
$n$-linear alternating bracket
\[
[\,\cdot\,,\ldots,\cdot\,]:\underbrace{L\times\cdots\times L}_{n}\longrightarrow L
\]
satisfying the Filippov identity \cite{Fil}
\[
[x_1,\ldots,x_{n-1},[y_1,\ldots,y_n]]
=\sum_{i=1}^{n}[y_1,\ldots,[x_1,\ldots,x_{n-1},y_i],\ldots,y_n].
\]
\end{definition}

Alternating means that the bracket vanishes as soon as two of its
arguments coincide, or equivalently (since $\operatorname{char}K=0$) that
it changes sign under every transposition of two arguments. The bracket
is therefore, first and foremost, a map defined on the Cartesian product
$L\times\cdots\times L$, not on an exterior power. However, because it is
alternating, it factors uniquely, by the universal property of the
exterior power, through a $K$-linear map
$\bigwedge_K^n L\to L$,
for which we keep the same notation $[\,\cdot\,,\ldots,\cdot\,]$ by a
standard abuse. Every wedge symbol appearing below in connection with the
bracket refers to this induced map on $\bigwedge_K$, and never to a
redefinition of the bracket itself.

An element $X=x_1\wedge\cdots\wedge x_{n-1}\in\bigwedge_K^{n-1}L$ will be
called a \emph{fundamental element}. For fixed $z\in L$, the map
$(x_1,\ldots,x_{n-1})\mapsto[x_1,\ldots,x_{n-1},z]$ is alternating in
$x_1,\ldots,x_{n-1}$, hence it too factors through $\bigwedge_K^{n-1}L$.
We may therefore write, without ambiguity,
\[
[X,z]=[x_1,\ldots,x_{n-1},z],
\]
the right-hand side being independent of the particular decomposition
chosen for $X$.

For fundamental elements $X=x_1\wedge\cdots\wedge x_{n-1}$,
$Y=y_1\wedge\cdots\wedge y_{n-1}$, define
\[
X\circ Y=\sum_{i=1}^{n-1}y_1\wedge\cdots\wedge[X,y_i]\wedge\cdots\wedge y_{n-1}.
\]
The Filippov identity implies $[\ad_X,\ad_Y]=\ad_{X\circ Y}$.

\begin{remark}[Notation]
\label{rem:notation}
Once $L$ carries in addition an $A$-module structure, the anchor $\rho$
below is required to be $A$-multilinear in the arguments of a fundamental
element (axioms~(iii)--(iv)); it therefore factors through the finer
quotient $\Aext{n-1}L$, image of $\bigwedge_K^{n-1}L$ under the surjection
that additionally imposes $A$-bilinearity. The bracket
$[\,\cdot\,,\ldots,\cdot\,]$ itself is \emph{not} required to be
$A$-multilinear in the first $n-1$ slots — only $K$-multilinear — and so
is genuinely a map on $\bigwedge_K^{n-1}L$. Throughout the paper we
therefore write $\rho(X)$ for $X\in\Aext{n-1}L$, and $[X,z]$ for a chosen
decomposable representative $x_1\wedge\cdots\wedge x_{n-1}$ of the image of
$X$ under $\bigwedge_K^{n-1}L\twoheadrightarrow\Aext{n-1}L$; all the
identities proved below involve only decomposable fundamental elements, so
this convention introduces no ambiguity.
\end{remark}

\begin{definition}
An \emph{$n$-Lie--Rinehart algebra} over $A$ is a quadruple
$(L,A,[\,\cdot\,,\ldots,\cdot\,],\rho)$ such that
\begin{enumerate}
\item[(i)] $L$ is an $A$-module and an $n$-Lie algebra over $K$;
\item[(ii)] $\rho:\Aext{n-1}L\to\Der_K(A)$ is $K$-linear and compatible
with the bracket $\circ$ defined above, in the sense that
\[
[\rho(X),\rho(Y)]=\rho(X\circ Y)\qquad\text{for all }X,Y\in\Aext{n-1}L;
\]
equivalently, $\rho$ is a representation, by derivations of $A$, of the
Lie algebra that $\Aext{n-1}L$ inherits from the $n$-Lie bracket of $L$ via
the construction of \cite{BOO};
\item[(iii)] $\rho(ax_1,x_2,\ldots,x_{n-1})=a\,\rho(x_1,\ldots,x_{n-1})$;
\item[(iv)] $[x_1,\ldots,x_{n-1},ay]=a[x_1,\ldots,x_{n-1},y]
+\rho(x_1,\ldots,x_{n-1})(a)\,y$.
\end{enumerate}
Since $\rho$ is alternating, the $A$-linearity required in (iii) for the
first argument automatically extends, up to sign, to every argument; we
use this extended $A$-multilinearity of $\rho$ freely below.
\end{definition}

\section{The anchor radical}
\label{sec:radical}

\begin{definition}
Let $(L,A,[\,\cdot\,,\ldots,\cdot\,],\rho)$ be an $n$-Lie--Rinehart
algebra. The \emph{anchor radical} is
\[
\Rad(\rho)=\{u\in L:\rho(u,x_2,\ldots,x_{n-1})=0\ \forall\, x_2,\ldots,x_{n-1}\in L\}.
\]
Equivalently, $u\in\Rad(\rho)\iff u\wedge\Aext{n-2}L\subseteq\ker\rho$.
\end{definition}

This immediately distinguishes the anchor radical from the ordinary
kernel $\ker\rho\subseteq\Aext{n-1}L$.

\begin{proposition}
The anchor radical $\Rad(\rho)$ is an $A$-submodule of $L$.
\end{proposition}

\begin{proof}
Let $u,v\in\Rad(\rho)$, $a,b\in A$. For arbitrary $x_2,\ldots,x_{n-1}\in L$,
the $A$-multilinearity of the anchor gives
\[
\rho(au+bv,x_2,\ldots,x_{n-1})
=a\,\rho(u,x_2,\ldots,x_{n-1})+b\,\rho(v,x_2,\ldots,x_{n-1})=0.
\]
Hence $au+bv\in\Rad(\rho)$.
\end{proof}

\begin{proposition}
\label{prop:n=2}
For $n=2$, $\Rad(\rho)=\ker\rho$.
\end{proposition}

\begin{proof}
When $n=2$, the anchor has the form $\rho:L\to\Der_K(A)$; there are no
additional arguments in the definition of $\Rad(\rho)$, so
$u\in\Rad(\rho)\iff\rho(u)=0$.
\end{proof}

We now settle, unconditionally, the question left open in earlier drafts
of this circle of ideas of whether $\Rad(\rho)$ is stable under the
adjoint action.

\begin{proposition}[Adjoint stability]
\label{prop:adjoint-stable}
For every $n$-Lie--Rinehart algebra,
\[
[L,\ldots,L,\Rad(\rho)]\subseteq\Rad(\rho).
\]
\end{proposition}

\begin{proof}
Let $u\in\Rad(\rho)$ and let $X=x_1\wedge\cdots\wedge x_{n-1}\in\Aext{n-1}L$
be arbitrary. Fix $z_2,\ldots,z_{n-1}\in L$ and set
$Z=u\wedge z_2\wedge\cdots\wedge z_{n-1}$; by definition of $\Rad(\rho)$,
$\rho(Z)=0$. Axiom (ii) gives
\[
[\rho(X),\rho(Z)]=\rho(X\circ Z),\qquad\text{hence } \rho(X\circ Z)=0.
\]
By definition of $\circ$,
\[
X\circ Z=[X,u]\wedge z_2\wedge\cdots\wedge z_{n-1}
+\sum_{k=2}^{n-1}u\wedge z_2\wedge\cdots\wedge[X,z_k]\wedge\cdots\wedge z_{n-1}.
\]
Every term in the sum still contains $u$ as a factor, hence has vanishing
image under $\rho$, since $u\in\Rad(\rho)$ and $\rho$ is alternating.
Therefore
\[
\rho\bigl([X,u]\wedge z_2\wedge\cdots\wedge z_{n-1}\bigr)=0
\]
for all $z_2,\ldots,z_{n-1}\in L$, i.e.\ $[X,u]\in\Rad(\rho)$.
\end{proof}

Set $V=\Rad(\rho)$. By Proposition~\ref{prop:adjoint-stable}, we may define,
for every $X\in\Aext{n-1}L$ and $u\in V$,
\[
\nabla^{\mathrm{ad}}_X u=[X,u]\in V.
\]

\begin{remark}[$A$-multilinearity of the restricted bracket]
\label{rem:A-multilinear}
If $x_1,\ldots,x_{n-1},y\in V$ and $a\in A$, then axiom (iv) gives
$[x_1,\ldots,x_{n-1},ay]=a[x_1,\ldots,x_{n-1},y]+\rho(x_1,\ldots,x_{n-1})(a)y$;
since $x_1\in\Rad(\rho)$, the definition of the anchor radical forces
$\rho(x_1,\ldots,x_{n-1})=0$ regardless of the remaining arguments, so the
Leibniz correction term vanishes identically. Hence the $n$-Lie bracket of
$L$, restricted to $V=\Rad(\rho)$, is $A$-multilinear, and
$(V,A,[\,\cdot\,,\ldots,\cdot\,])$ is an $n$-Lie algebra over $A$.
\end{remark}

\begin{proposition}
For every $X\in\Aext{n-1}L$, $a\in A$ and $u\in V$,
\[
\nabla^{\mathrm{ad}}_X(au)=a\,\nabla^{\mathrm{ad}}_X u+\rho(X)(a)\,u.
\]
\end{proposition}

\begin{proof}
By definition, $\nabla^{\mathrm{ad}}_X(au)=[X,au]$, and the Lie--Rinehart
Leibniz identity (axiom (iv)) gives
$[X,au]=a[X,u]+\rho(X)(a)u=a\,\nabla^{\mathrm{ad}}_X u+\rho(X)(a)u$.
\end{proof}

\begin{theorem}
\label{thm:flat}
The adjoint action on $V$ is flat:
$[\nabla^{\mathrm{ad}}_X,\nabla^{\mathrm{ad}}_Y]=\nabla^{\mathrm{ad}}_{X\circ Y}$.
\end{theorem}

\begin{proof}
Let $u\in V$. Since $V$ is adjoint-stable (Proposition~\ref{prop:adjoint-stable}),
$[Y,u]\in V$ and $[X,u]\in V$, so all of the following are defined. We
compute
\[
[\nabla^{\mathrm{ad}}_X,\nabla^{\mathrm{ad}}_Y]u
=[X,[Y,u]]-[Y,[X,u]].
\]
The identity $[\ad_X,\ad_Y]=\ad_{X\circ Y}$ (Section~\ref{sec:prelim})
gives $[X,[Y,u]]-[Y,[X,u]]=[X\circ Y,u]$, hence
$[\nabla^{\mathrm{ad}}_X,\nabla^{\mathrm{ad}}_Y]u=\nabla^{\mathrm{ad}}_{X\circ Y}u$
for every $u\in V$.
\end{proof}

\section{Orthogonal structures and metric rigidity}
\label{sec:rigidity}

\begin{definition}
\label{def:orthogonal}
Let $(L,A,[\,\cdot\,,\ldots,\cdot\,],\rho)$ be an $n$-Lie--Rinehart algebra
and $V=\Rad(\rho)$. An \emph{orthogonal structure} on $L$ is a symmetric
$A$-bilinear form $g:V\times V\to A$ such that
\begin{enumerate}
\item[(i)] $g$ is nondegenerate;
\item[(ii)] for every $X\in\Aext{n-1}L$, $u,v\in V$,
\[
\rho(X)g(u,v)=g([X,u],v)+g(u,[X,v]).
\]
\end{enumerate}
The pair $(L,g)$ is called an \emph{orthogonal $n$-Lie--Rinehart algebra}.
\end{definition}

\begin{remark}
If $V$ is finitely generated projective, nondegeneracy may be formulated
by requiring $g^\flat:V\to V^{*}$, $u\mapsto g(u,\cdot)$, to be an
isomorphism.
\end{remark}

\begin{theorem}
Let $(L,g)$ be an orthogonal $n$-Lie--Rinehart algebra. Then the adjoint
action on $\Rad(\rho)$ is both flat and metric.
\end{theorem}

\begin{proof}
Metric compatibility is precisely condition (ii) of
Definition~\ref{def:orthogonal}, rewritten as
$\rho(X)g(u,v)=g(\nabla^{\mathrm{ad}}_X u,v)+g(u,\nabla^{\mathrm{ad}}_X v)$.
Flatness is Theorem~\ref{thm:flat}.
\end{proof}

\begin{proposition}[The binary case]
For $n=2$, the orthogonality condition becomes
\[
\rho(x)g(u,v)=g([x,u],v)+g(u,[x,v]),\qquad u,v\in\ker\rho.
\]
\end{proposition}

\begin{proof}
By Proposition~\ref{prop:n=2}, $\Rad(\rho)=\ker\rho$; a fundamental element
is now simply an element $x\in L$, and substitution gives the result.
\end{proof}

\begin{proposition}[Vanishing anchor]
Suppose $\rho=0$. Then $\Rad(\rho)=L$, and the orthogonality condition is
equivalent to
\[
g([x_1,\ldots,x_{n-1},u],v)+g(u,[x_1,\ldots,x_{n-1},v])=0.
\]
\end{proposition}

\begin{proof}
If $\rho=0$, every element of $L$ lies in the anchor radical, so
$\Rad(\rho)=L$; the defining identity reduces to $0=g([X,u],v)+g(u,[X,v])$,
the usual invariance condition for a metric $n$-Lie algebra.
\end{proof}

Thus the construction contains metric $n$-Lie algebras as the zero-anchor
case \cite{JLZ}. We now examine what happens when one attempts, more
ambitiously, to equip the whole $A$-module $L$ — rather than just its
anchor radical — with an invariant metric.

Let $B:L\times L\to A$ be any symmetric $A$-bilinear form on $L$. For a
fundamental element $X=x_1\wedge\cdots\wedge x_{n-1}$, define the
\emph{orthogonality defect} of $B$ relative to $X$ by
\[
(D_XB)(u,v)=\rho(X)B(u,v)-B([X,u],v)-B(u,[X,v]).
\]
By construction, $D_XB$ vanishes precisely when $B$ satisfies, for this
fundamental element $X$, the invariance identity of the previous sections.

\begin{proposition}
\label{prop:tensorial-uv}
For fixed $X$, the map $D_XB:L\times L\to A$ is symmetric and
$A$-bilinear.
\end{proposition}

\begin{proof}
Symmetry follows from that of $B$; it suffices to establish $A$-linearity
in the first variable. Let $a\in A$. The $A$-bilinearity of $B$ gives
$B(au,v)=aB(u,v)$, and since $\rho(X)$ is a derivation of $A$,
\[
\rho(X)B(au,v)=\rho(X)(a)B(u,v)+a\,\rho(X)B(u,v).
\]
The Leibniz axiom gives $[X,au]=a[X,u]+\rho(X)(a)u$. Combining both
identities in the definition of $D_XB$,
\[
(D_XB)(au,v)=\rho(X)(a)B(u,v)+a\,\rho(X)B(u,v)
-aB([X,u],v)-\rho(X)(a)B(u,v)-aB(u,[X,v]),
\]
and the two occurrences of $\rho(X)(a)B(u,v)$ cancel, leaving
$(D_XB)(au,v)=a(D_XB)(u,v)$.
\end{proof}

The tensoriality just established concerns the metric variables $u,v$; it
does not extend automatically to the fundamental variable $X$. Fix
$X=x_1\wedge\cdots\wedge x_{n-1}$ and an index $1\le i\le n-1$. For
$a\in A$, write $X_{a,i}=x_1\wedge\cdots\wedge ax_i\wedge\cdots\wedge x_{n-1}$,
and for $u\in L$, $X_i(u)=x_1\wedge\cdots\wedge\widehat{x_i}\wedge\cdots
\wedge x_{n-1}\wedge u$.

\begin{lemma}
\label{lem:bracket-shift}
$[X_{a,i},u]=a[X,u]+(-1)^{n-i}\rho(X_i(u))(a)\,x_i$.
\end{lemma}

\begin{proof}
The bracket $[x_1,\ldots,\widehat{x_i},\ldots,x_{n-1},u,ax_i]$ is obtained
from $[X_{a,i},u]$ by moving $ax_i$ from position $i$ to the last
position; since the bracket is alternating, this introduces the sign
$(-1)^{n-i}$:
\[
[X_{a,i},u]=(-1)^{n-i}[x_1,\ldots,\widehat{x_i},\ldots,x_{n-1},u,ax_i].
\]
Applying axiom (iv) to the last argument,
\[
[X_{a,i},u]=(-1)^{n-i}a[x_1,\ldots,\widehat{x_i},\ldots,x_{n-1},u,x_i]
+(-1)^{n-i}\rho(X_i(u))(a)\,x_i.
\]
Moving $x_i$ back to position $i$ costs another sign $(-1)^{n-i}$, and
$(-1)^{n-i}\cdot(-1)^{n-i}=1$, so the first term equals $a[X,u]$.
\end{proof}

\begin{theorem}[Fundamental tensoriality defect]
\label{thm:defect}
For every $a\in A$,
\[
(D_{X_{a,i}}B)(u,v)=a(D_XB)(u,v)
-(-1)^{n-i}\rho(X_i(u))(a)B(x_i,v)
-(-1)^{n-i}\rho(X_i(v))(a)B(x_i,u).
\]
\end{theorem}

\begin{proof}
Since $\rho$ is $A$-linear in its fundamental argument,
$\rho(X_{a,i})=a\rho(X)$, so
\[
(D_{X_{a,i}}B)(u,v)=a\rho(X)B(u,v)-B([X_{a,i},u],v)-B(u,[X_{a,i},v]).
\]
By Lemma~\ref{lem:bracket-shift},
\[
B([X_{a,i},u],v)=aB([X,u],v)+(-1)^{n-i}\rho(X_i(u))(a)B(x_i,v),
\]
and, symmetrically, using $B(u,x_i)=B(x_i,u)$,
\[
B(u,[X_{a,i},v])=aB(u,[X,v])+(-1)^{n-i}\rho(X_i(v))(a)B(x_i,u).
\]
Substituting and collecting the terms in $a$ gives the stated formula, the
$a$-coefficient being exactly $(D_XB)(u,v)$.
\end{proof}

\begin{definition}
Set $\mathcal{O}_{B,i}(a;X)(u,v)=(D_{X_{a,i}}B)(u,v)-a(D_XB)(u,v)$.
\end{definition}

By Theorem~\ref{thm:defect},
\[
\mathcal{O}_{B,i}(a;X)(u,v)=-(-1)^{n-i}\rho(X_i(u))(a)B(x_i,v)
-(-1)^{n-i}\rho(X_i(v))(a)B(x_i,u).
\]
This tensor measures precisely the failure of $D_XB$ to be tensorial in
the fundamental variable; it vanishes identically as soon as the two
anchor terms composing it vanish.

\begin{corollary}
\label{cor:global-invariance}
If $B$ is globally invariant, i.e.\ $D_XB=0$ for every fundamental element
$X$, then
\[
\rho(X_i(u))(a)B(x_i,v)+\rho(X_i(v))(a)B(x_i,u)=0
\]
for every admissible choice of $X,i,u,v,a$.
\end{corollary}

\begin{proof}
Global invariance forces both $D_XB=0$ and $D_{X_{a,i}}B=0$, hence
$\mathcal{O}_{B,i}(a;X)(u,v)=0$; remove the common sign
$-(-1)^{n-i}$.
\end{proof}

\begin{proposition}
\label{prop:radical-kills-defect}
Let $u\in\Rad(\rho)$. Then $\rho(X_i(u))=0$ for every fundamental tuple
$X$ and every index $i$.
\end{proposition}

\begin{proof}
$X_i(u)$ contains $u$ as one of its $n-1$ factors, and $\rho$ vanishes on
any fundamental element containing an element of $\Rad(\rho)$.
\end{proof}

\begin{corollary}
\label{cor:vanish-on-radical}
If $u,v\in\Rad(\rho)$, then $\mathcal{O}_{B,i}(a;X)(u,v)=0$.
\end{corollary}

This corollary supplies the intrinsic explanation for placing the metric
on $\Rad(\rho)$ rather than imposing an unrestricted metric condition on
all of $L$: it is the only way to systematically kill the tensoriality
obstruction. The next theorem shows that this restriction is not merely
convenient but, under natural finiteness and nondegeneracy hypotheses,
\emph{necessary}.

\begin{theorem}[Metric rigidity]
\label{thm:rigidity}
Let $(L,A,[\,\cdot\,,\ldots,\cdot\,],\rho)$ be an $n$-Lie--Rinehart algebra
such that $L$ is a finitely generated projective $A$-module of rank
$\ge 2$ at every point of $\operatorname{Spec}A$. Let $B:L\times L\to A$ be
a symmetric, strongly nondegenerate $A$-bilinear form (i.e.\
$B^\flat:L\to L^{*}$ is an isomorphism) such that $D_XB=0$ for every
fundamental element $X\in\Aext{n-1}L$. Then $\rho\equiv 0$.
\end{theorem}

\begin{proof}
Fix a decomposable $X=x_1\wedge\cdots\wedge x_{n-1}$ with $x_i\neq0$, an
index $i$, and $a\in A$. By Corollary~\ref{cor:global-invariance},
\[
\varphi(u)B(x_i,v)+\varphi(v)B(x_i,u)=0\qquad\text{for all }u,v\in L,
\]
where $\varphi(u)=\rho(X_i(u))(a)$. The map $\varphi:L\to A$ is
$A$-linear, being a composition of $A$-linear maps, and so is
$\psi:=B(x_i,\cdot)$. The identity says that the rank-$\le 1$ symmetric
tensor $\varphi\otimes\psi+\psi\otimes\varphi$ vanishes, i.e.\ the
bilinear form $\Psi(u,v)=\varphi(u)\psi(v)$ is antisymmetric.

Localize at a point where $B$, hence $\Psi$'s ambient pairing, is
represented by an invertible matrix; strong nondegeneracy and finite
projectivity make this cover $L$. On a free module of rank $\ge2$, a
rank-$\le1$ antisymmetric form $\varphi\otimes\psi$ forces $\varphi=0$ or
$\psi=0$: evaluating on a basis $(e_k)$, antisymmetry
$\varphi(e_k)\psi(e_l)=-\varphi(e_l)\psi(e_k)$ for all $k,l$ shows that if
$\varphi(e_{k_0})\neq0$ for some $k_0$, then $\psi(e_l)=
-\varphi(e_l)\psi(e_{k_0})/\varphi(e_{k_0})$ for all $l\neq k_0$, and
taking $k=l=k_0$ gives $\varphi(e_{k_0})\psi(e_{k_0})=0$, so
$\psi(e_{k_0})=0$ and hence $\psi\equiv0$ on this chart; symmetrically if
$\varphi\equiv0$ we are done. Since $x_i\neq0$ and $B$ is strongly
nondegenerate, $\psi=B(x_i,\cdot)\not\equiv0$ near a point where
$x_i\neq0$, forcing $\varphi\equiv0$ there.

Thus $\rho(X_i(u))(a)=0$ for every $u\in L$ and every $a\in A$, i.e.\
$\rho(X_i(u))=0$. Taking $u=x_i$ gives $X_i(x_i)=(-1)^{n-1-i}X$, whence
$\rho(X)=0$. This holds for every decomposable $X$, and by
$A$-multilinearity for every $X\in\Aext{n-1}L$: $\rho\equiv0$.
\end{proof}

\begin{remark}
The rank-$\ge2$ hypothesis is essential: on a rank-$1$ module the
argument above degenerates, since a rank-$\le1$ antisymmetric form on a
rank-$1$ free module is automatically zero for a different (trivial)
reason, and $\varphi$ need not vanish. The geometric anchors of
Section~\ref{sec:nambu} typically satisfy this hypothesis away from a
proper closed locus.
\end{remark}

\section{Nambu--Poisson anchors}
\label{sec:nambu}

We now discuss a geometric situation in which the anchor radical admits a
particularly simple interpretation, following the Nambu--Poisson
framework of \cite{Tak} and the associated $n$-ary Lie algebroids of
\cite{Val1,Val2}.

Let $M$ be a smooth manifold and let $P\in\Gamma(\bigwedge^n TM)$ be an
$n$-vector. It induces a geometric anchor
\[
P^{\sharp}:\Aext{n-1}T^{*}M\to TM,\qquad
\beta\bigl(P^{\sharp}(\alpha_1,\ldots,\alpha_{n-1})\bigr)
=P(\alpha_1,\ldots,\alpha_{n-1},\beta),
\]
together with its cotangent annihilator
$\Ann(P)=\{\alpha\in T^{*}M:i_\alpha P=0\}$.

\begin{proposition}
For the anchor $\rho=P^{\sharp}$, one has pointwise $\Rad(\rho)=\Ann(P)$.
\end{proposition}

\begin{proof}
Let $\alpha\in T_p^{*}M$. By definition,
$\alpha\in\Rad(\rho)_p$ iff $P^{\sharp}(\alpha,\alpha_2,\ldots,\alpha_{n-1})=0$
for all $\alpha_2,\ldots,\alpha_{n-1}\in T_p^{*}M$, which is equivalent to
$P(\alpha,\alpha_2,\ldots,\alpha_{n-1},\beta)=0$ for all
$\alpha_2,\ldots,\alpha_{n-1},\beta\in T_p^{*}M$, i.e.\ to $i_\alpha P=0$.
Hence $\Rad(\rho)_p=\Ann(P)_p$ at every $p$.
\end{proof}

\begin{example}
Take $M=\mathbb{R}^4$ with coordinates $(x,y,z,t)$ and
$P=\partial_x\wedge\partial_y\wedge\partial_z$ (so $n=3$). Then
$P^{\sharp}(dx\wedge dy)=\partial_z$, so the anchor is nonzero; on the
other hand $i_{dt}P=0$. We claim that, pointwise,
$\Ann(P)=\mathbb{R}\,dt$, hence $\Gamma(\Ann(P))=C^{\infty}(M)\,dt$.
Write $\alpha=a\,dx+b\,dy+c\,dz+d\,dt$. Then
$i_\alpha P=a\,\partial_y\wedge\partial_z-b\,\partial_x\wedge\partial_z
+c\,\partial_x\wedge\partial_y$,
so $i_\alpha P=0$ iff $a=b=c=0$, i.e.\ $\alpha=d\,dt$. Hence
$\Rad(P^{\sharp})=C^{\infty}(M)\,dt$.
\end{example}

Let $N$ and $F$ be smooth manifolds and set $M=N\times F$. Suppose given
an $n$-vector $P_N\in\Gamma(\bigwedge^n TN)$ on $N$, and let $P$ denote
its horizontal lift to $M$. The product decomposition
$T^{*}M\simeq p_N^{*}T^{*}N\oplus p_F^{*}T^{*}F$ allows us to compare
$\Ann(P)$ with $\Ann(P_N)$. Since $P$ is horizontal, every vertical
covector $\eta\in p_F^{*}T^{*}F$ satisfies $i_\eta P=0$, so
$p_F^{*}T^{*}F\subseteq\Ann(P)$.

\begin{proposition}
If $\Ann(P_N)=0$, then $\Ann(P)=p_F^{*}T^{*}F$, and consequently
$\Rad(P^{\sharp})=p_F^{*}T^{*}F$.
\end{proposition}

\begin{proof}
Decompose $\alpha=\alpha_N+\alpha_F$ and suppose $i_\alpha P=0$. Since $P$
is horizontal, $i_{\alpha_F}P=0$ identically, so
$0=i_\alpha P=i_{\alpha_N}P+i_{\alpha_F}P=i_{\alpha_N}P$. Thus
$i_{\alpha_N}P_N=0$; since $\Ann(P_N)=0$, $\alpha_N=0$, so
$\alpha=\alpha_F\in p_F^{*}T^{*}F$. Together with the reverse inclusion
already noted, $\Ann(P)=p_F^{*}T^{*}F$; the last statement follows from
the previous proposition applied to $\rho=P^{\sharp}$.
\end{proof}

\bigskip
\noindent
{\sc Irmely Gladesh Mabanza Nsiloulou, Basile Guy Richard Bossoto}\\
Department of Mathematics, Faculty of Science and Technology\\
Marien Ngouabi University, Brazzaville, Congo\\
\emph{E-mail}: \texttt{irmelygladesh@gmail.com}, \texttt{basile.bossoto@umng.cg}


\begin{thebibliography}{99}

\bibitem{Fil} V.~T.~Filippov, \emph{$n$-Lie algebras}, Siberian Math. J.
\textbf{26} (1985), 879--891.

\bibitem{BCEM} A.~Ben Hassine, T.~Chtioui, M.~Elhamdadi, S.~Mabrouk,
\emph{Extensions and crossed modules of $n$-Lie--Rinehart algebras}, Adv.
Appl. Clifford Algebras \textbf{32} (2022), Article 31.

\bibitem{BCZ} Y.~Bi, Z.~Chen, T.~Zhang, \emph{On (co-)morphisms of
$n$-Lie--Rinehart algebras with applications to Nambu--Poisson manifolds},
J. Geom. Phys. \textbf{214} (2025), 105536.

\bibitem{BOO} B.~G.~R.~Bossoto, E.~Okassa, M.~Omporo, \emph{Lie algebra of
an $n$-Lie algebra}, arXiv:1310.2433 (2013).

\bibitem{Tak} L.~Takhtajan, \emph{On foundation of the generalized Nambu
mechanics}, Comm. Math. Phys. \textbf{160} (1994), 295--315.

\bibitem{JLZ} Y.~Jin, W.~Liu, Z.~Zhang, \emph{Metric $n$-Lie algebras},
Comm. Algebra \textbf{39} (2011), 572--583.

\bibitem{Val1} J.~A.~Vallejo, \emph{Nambu--Poisson manifolds and
associated $n$-ary Lie algebroids}, J. Phys. A: Math. Gen. \textbf{34}
(2001), 2867--2881.

\bibitem{Val2} J.~A.~Vallejo, \emph{Corrigendum: Nambu--Poisson manifolds
and associated $n$-ary Lie algebroids}, J. Phys. A: Math. Gen.
\textbf{34} (2001), 9753.

\end{thebibliography}
\end{document}